\documentclass{amsart}

\usepackage{amssymb,color}
\usepackage{amsfonts}
\usepackage{amsmath}
\usepackage{euscript}
\usepackage{enumerate}
\usepackage{graphics}

\newtheorem{theorem}{Theorem}[section]
\newtheorem{lemma}[theorem]{Lemma}
\newtheorem{note}[theorem]{Note}
\newtheorem{prop}[theorem]{Proposition}
\newtheorem{cor}[theorem]{Corollary}
\newtheorem{exa}[theorem]{Example}

\newtheorem*{Theorem1'}{Theorem 1'}

\theoremstyle{definition}

\theoremstyle{remark}

\numberwithin{equation}{section}

\newcommand \C{{\mathbb C}}

\newcommand \End{{\mathrm {End}}}
\newcommand \Hom{{\mathrm {Hom}}}
\newcommand \dm{{\mathrm {dim}}}

\newcommand \R{{\mathbb R}}

\newcommand \GL{{\mathrm {GL}}}

\newcommand \FGL{{\mathrm {FGL}}}

\newcommand \Aut{{\mathrm{ Aut}}}

\newcommand \ind{{\mathrm {ind}}}
\newcommand \res{{\mathrm {res}}}

\begin{document}

\title [\small{Clifford theory for infinite dimensional modules}] {\small{Clifford theory for infinite dimensional modules}}

\author{Fernando Szechtman}
\address{Department of Mathematics and Statistics, University of Regina}
\email{fernando.szechtman@gmail.com}
\thanks{The author was supported in part by an NSERC discovery grant}





\begin{abstract} Clifford theory of possibly infinite dimensional
modules is studied.
\end{abstract}

\maketitle

\section{Introduction}

Let $F$ be a field and let $N\unlhd G$ be groups. In 1937 Clifford
\cite{C} laid down a strategy for studying the finite dimensional
irreducible $FG$-modules in a two-step program that works under
the assumption that $[G:N]$ be finite. In the first place, every
such a module, say $V$, is isomorphic to $\ind_T^G S$, where $T$
is the inertia group of an irreducible $FN$-submodule $W$ of $V$
and $S$ is an irreducible $FT$-module, namely the $W$-homogeneous
component of $\res_N^G V$. Moreover, $S\mapsto \ind_T^G S$ yields
a bijective correspondence between the isomorphism classes of
finite dimensional irreducible modules of $FT$ and $FG$ lying over
$W$. This reduces the above study to the case $T=G$. For this
stage Clifford assumed that $F$ is algebraically closed.
Accordingly, there is a projective representation $X:G\to\GL(W)$
extending $R$ with factor set $\alpha\in Z^2(G,F^*)$ constant on
cosets of $N$. Set $\beta=\alpha^{-1}$ with corresponding
$\delta\in Z^2(G/N,F^*)$. Then $U\mapsto U\otimes_F W$ yields a
bijective correspondence between the isomorphism classes of
irreducible modules of the twisted group algebra $F^\delta [G/N]$
and the isomorphism classes of finite dimensional irreducible
modules of $FG$ lying over $W$. Here $U$ affords the irreducible
projective representation~$Y:G\to\GL(U)$ defined by~$F^\beta
[G]\to F^\delta [G/N]\to\GL(U)$ and $U\otimes_F W$ affords the
irreducible representation $R:G\to\GL(U\otimes_F W)$ given by
$R(g)=Y(g)\otimes X(g)$, $g\in G$.

Several authors have extended one or both of the above
correspondences to a much wider setting. The first extension, due
to Mackey \cite{M2}, involved unitary representations of locally
compact groups. About a decade later, Dade \cite{D}, Fell
\cite{F}, Ward \cite{W} and Cline \cite{Cl} extended Clifford's
ideas to an axiomatic theory of group-graded rings called Clifford
systems. A further abstraction by Dade \cite{D2} became
essentially categorical in nature. A purely categorical view of
the theory is taken by Alperin \cite{A}, Galindo \cite{G,G2} and
Lizasoain \cite{Li}. On the other hand, a Clifford correspondence
from the point of view of normal subrings of a ring was studied by
Rieffel \cite{R}, and a Clifford theory of Hopf algebras was
investigated Schneider \cite{Sc}, van Oystaeyen and Zhang
\cite{VZ}, Witherspoon \cite{Wi} and Burciu \cite{B}. A more
recent paper by Witherspoon \cite{Wi2} unifies some prior work on
the subject and furnishes a number of versions of Clifford
correspondence for finite dimensional modules of an associative
algebra over an algebraically closed field.


As far as we know, the infinite dimensional purely algebraic case
of the above correspondences has not been considered so far. One
possible reason for this is that Clifford theory breaks down at
the very outset: an infinite dimensional $FG$-module may lack
irreducible $FN$-submodules. A generic way to produce such
examples can be found in \S\ref{sec:ex1}. This phenomenon is
impossible when $[N:1]$ or $[G:N]$ is finite (cf. \cite[Theorem
7.2.16]{P}). It is also impossible if one considers
representations $R:G\to\FGL(V)$, where $\FGL(V)$ is the subgroup
of $\GL(V)$ of all finitary automorphisms of $V$. This follows
from the following general result due independently to
Meierfrankenfeld \cite{Me} and  Wehrfritz \cite{We} (who does not
require $F$ to be commutative): if $G$ is an irreducible subgroup
of $\FGL(V)$ and $H$ is an ascentant subgroup (this a
generalization of the notion of subnormal subgroup) of $G$ then
$H$ is completely reducible. Further results on primitive and
irreducible subgroups of $\FGL(V)$, with $V$ infinite dimensional,
can be found in \cite{Me}, \cite{We2} as well as in \cite{We3} and
references therein.

In this paper (cf. \S\ref{sec:cli}) we extend both Clifford
correspondences to the infinite dimensional case under bare
minimum assumptions. As corollaries we obtain the infinite
dimensional analogue of Gallagher's theorem \cite[Theorem 2]{G}
describing the complex irreducible characters of $G$ lying over a
given irreducible character of $N$ extendible to $G$, as well as
the infinite dimensional analogue of the well-known (cf. \cite[\S
3]{S}) description of the complex irreducible characters of the
direct product of two groups.

As an application that makes essential use of most of the results
in this paper, \cite{SHI} studies representations of McLain's
groups $M=M(\Lambda, \leq, R)$, where $(\Lambda,\leq)$ is a
totally ordered set and $R$ is a fairly general type of ring (see
\cite{Ma} for the original use of these groups). We stress the
fact that $\Lambda$ may be infinite, $R$ need not be finite or
commutative, the representations of $M$ are allowed to be infinite
dimensional, and very mild restrictions are imposed on the
underlying field. Nothing seems to be known about the
representation theory of $M$ in this generality. In the special
case when $|\Lambda|=n$ is finite and $R=F_q$ is a finite field,
then $M=U_n(q)$. The results of this paper are used in \cite{SHI}
to extend to arbitrary $M$ many of the results concerning the
complex supercharacters of $U_n(q)$, first encountered by Lehrer
\cite{Le} and further developed by Andr\'{e} \cite{An}. The
subject has attracted considerable attention since~\cite{An},
specially after its axiomatization by Diaconis and Isaacs
\cite{DI}. A second application of our results appears in \cite{Sz}.

It should also be noted that Mackey's decomposition theorem and
tensor product theorem (cf. \cite[Theorems 1 and 2]{M}) are valid
under no assumptions whatsoever. For finite groups, the latter is
usually obtained as a consequence of the former with the aid of
some auxiliary result (Lemma 2 in \cite{M}, Theorem 43.2 in
\cite{CR}, and Theorem 5.17.3 in \cite{J}). We begin the paper
(cf. \S\ref{sec:ma}) by furnishing, under no restrictions, short
proofs of both \cite[Theorems 1 and 2]{M} which highlight the very
reason why they are true: to determine the structure of an
imprimitive module, just find the orbits, the stabilizers and the
nature of the modules the latter stabilize.

In \S\ref{sec:ex2} we provide examples of a curious phenomenon
related to Clifford theory: it is
possible for $\ind_N^G W$ to be irreducible with inertia group
$I_G(W)=G\neq N$. This is in stark contrast to the well-known
irreducibility criterion valid for finite groups and complex representations (cf. \cite[\S 7]{S}).
We also furnish sufficient conditions for $I_G(W)=G$ and the irreducibility of $\ind_N^G W$ to force $G=N$.

All rings have a unit, shared by all their subrings. All modules
are left modules unless explicitly mentioned otherwise. Groups may
be infinite and representations are allowed to be infinite
dimensional.

\section{Mackey decomposition and tensor product theorems}\label{sec:ma}

We fix an arbitrary field $F$ throughout this section.

\begin{theorem}\label{m2} (Mackey Tensor Product Theorem) Let $G$ be a group with subgroups $H_1$ and
$H_2$, and let $V_1$ and $V_2$ be modules for
$H_1$ and $H_2$ over~$F$, respectively.

For $x,y\in G$ set $H^{(x,y)}=xH_1x^{-1}\cap yH_2y^{-1}$ and let
$V_1^x$ and $V_2^y$ denote the $F$-spaces $V_1$ and $V_2$,
respectively viewed as $FH^{(x,y)}$-modules via:
$$
h\cdot v_1=x^{-1}hx v_1,\quad h\in H^{(x,y)}, v_1\in V_1,
$$
$$
h\cdot v_2=y^{-1}hy v_2,\quad h\in H^{(x,y)}, v_2\in V_2.
$$
Then the $FG$-module $\ind_{H_1}^G V_1\otimes \ind_{H_2}^G V_2$
has the following decomposition:
$$
\ind_{H_1}^G V_1\otimes \ind_{H_2}^G V_2\cong \underset{d\in
D}\bigoplus V(d),
$$
where $D$ is a system of representatives for the
$(H_1,H_2)$-double cosets in $G$ and
$$
V(d)\cong \ind_{H^{(x,y)}}^G V_1^x\otimes V_2^y,\quad d\in D,
$$
for any choice of $(x,y)\in G\times G$ such that
$H_1x^{-1}yH_2=H_1dH_2$ (any two choices yield isomorphic $FG$-modules).
In particular,
$$
\ind_{H_1}^G V_1\otimes \ind_{H_2}^G V_2\cong \underset{y\in
D}\bigoplus \ind_{H_1\cap yH_2y^{-1}}^G V_1\otimes V_2^y.
$$
\end{theorem}

\begin{proof} We have
$$
\ind_{H_1}^G V_1=\underset{xH_1\in G/H_1}\bigoplus xV_1,\quad
\ind_{H_2}^G V_2=\underset{yH_2\in G/H_2}\bigoplus yV_2,
$$
where $G/H_i$ is the set of all left cosets of $H_i$ in $G$,
$1\leq i\leq 2$. Therefore,
$$
\ind_{H_1}^G V_1\otimes \ind_{H_2}^G V_2=\underset{(xH_1,yH_2)\in
G/H_1\times G/H_2}\bigoplus xV_1\otimes yV_2.
$$
The function $G/H_1\times G/H_2\to H_1\!\!\setminus\! G/H_2$ given by
$(xH_1,yH_2)\mapsto H_1x^{-1}yH_2$ is clearly surjective and its
fibers are the $G$-orbits of $G/H_1\times G/H_2$ under
componentwise left multiplication, that is, $(x_1H_1,y_1H_2)$ and
$(x_2H_1,y_2H_2)$ are in the same $G$-orbit if and only if
$$
 H_1x_1^{-1}y_1H_2=H_1x_2^{-1}y_2H_2.
$$
Moreover, the stabilizer in $G$ of a given $(xH_1,yH_2)\in
G/H_1\times G/H_2$ is $H^{(x,y)}$ and $V_1^x\otimes V_2^y\cong xV_1\otimes yV_2$
as $H^{(x,y)}$-modules via $v_1\otimes v_2\mapsto xv_1\otimes yv_2$.
\end{proof}

\begin{theorem}\label{m1} (Mackey Decomposition Theorem) Let $G$ be a group with subgroups $H$ and $K$ and let $V$ be an
$FH$-module. Then
$$
\res_K^G\, \ind_H^G V\cong\underset{x\in D}\bigoplus \ind_{K\cap
xHx^{-1}}^K V^x,
$$
where $D$ is a system of representatives for the
$(K,H)$-double cosets in $G$ and $V^x$ is the $F$-vector space $V$ being acted upon by $K\cap
xHx^{-1}$ as follows:
$$
h\cdot v=x^{-1}hx v,\quad h\in K\cap xHx^{-1}, v\in V.
$$
\end{theorem}

\begin{proof} We have
$$
\ind_{H}^G V=\underset{xH\in G/H}\bigoplus xV,
$$
where $G/H$ is the set of all left cosets of $H$ in $G$. When $K$ acts on $G/H$ by left multiplication, $xH$ and $yH$  are
in the same $K$-orbit if and only if
$$
KxH=KyH.
$$
Moreover, the stabilizer in $K$ of a given $xH\in
G/H$ is $K\cap xHx^{-1}$ and, in addition, $V^x\cong xV$ as $K\cap xHx^{-1}$-modules via $v\mapsto xv$.
\end{proof}

\begin{cor} Let $G$ be a group with subgroups $H_1$ and
$H_2$. Let $V_2$ be an $FH_2$-module and let $P$ be the
permutation $FG$-module arising from the coset space $G/H_1$. Then
$$
P\otimes \ind_{H_2}^G V_2\cong \ind_{H_1}^G \res_{H_1}^G \ind_{H_2}^G V_2.
$$
\end{cor}

\begin{proof} First use Theorem \ref{m2} with $V_1$ trivial, then the transitivity of induction
and the compatibility of induction with direct sums, and finally apply Theorem \ref{m1}.
\end{proof}

Suppose $H\leq G$ and let $W$ (resp. $V$) is a module for $FH$
(resp. $FG$). Then
\begin{equation}
\label{resext} (\ind_H^G W)\otimes V\cong \ind_H^G (W\otimes
\res_H^G V),
\end{equation}
a fairly well-known result in the finite dimensional case. As for
the proof, we have
$$
(\ind_H^G W)\otimes V\cong (\underset{xH\in G/H}\bigoplus
xW)\otimes V\cong \underset{xH\in G/H}\bigoplus (xW\otimes V),
$$
where $G$ transitively permutes the $xW\otimes V$ and the
stabilizer of $W\otimes V$ is $H$.

When $[G:H_1]$ and $[G:H_2]$ are finite, \cite{J} proves Theorem
\ref{m2} by means of Theorem \ref{m1} and (\ref{resext}), although
only a restricted version of (\ref{resext}) is considered, with a
somewhat longer proof. Incidently, a very special case of
(\ref{resext}) appears in \cite[Proposition 1.1]{L}, again with a
fairly longer proof.

\section{Clifford Theory}\label{sec:cli}

We fix an arbitrary field $F$ throughout this section.

\begin{lemma}\label{u1} Let $G$ be a group acting on a group~$N$ via
automorphisms. Let $W$ be an $FN$-module. For $g\in G$, consider
the $FN$-module~$W^g$, whose underlying $F$-vector space is $W$,
acted upon by $N$ as follows:
$$
x\cdot w= {}^g x w,\quad  x\in N,w\in W.
$$
Then
$$
I_G(W)=\{g\in G\,|\, W^g\cong W\}
$$
is subgroup of $G$. Moreover, if $W$ is isomorphic to an $FN$-module $U$, then
$W^g\cong U^g$ for all $g\in G$.
\end{lemma}

\begin{proof} Immediate consequence of the following:
if $g\in G$ and $f:W\to U$ is an $FN$-isomorphism, then so is
$f:W^g\to U^g$.
\end{proof}

We refer to $I_G(W)$ as the inertia group of $W$. If $I_G(W)=G$ we will say that $W$ is $G$-invariant.
We will mainly use this when $N\unlhd G$ and $G$ acts on $N$ by conjugation: $${}^g x=gxg^{-1}.$$

Given a ring $R$, an $R$-module $V$ and an irreducible submodule
$W$ of $V$, by the $W$-homogeneous component of $V$ we understand
the sum of all irreducible submodules of $V$ isomorphic to $W$.

\begin{theorem}\label{cl1} (Clifford) Let $N\unlhd G$ be groups and let
$V$ be an irreducible $FG$-module. Suppose that $V$ admits an
irreducible $FN$-submodule $W$. Then

(a) $V=\underset{g\in G}\sum gW$ is a completely reducible
$FN$-module.

(b) $\res_N^G V$ is the direct sum of its homogeneous components
and $G$ acts transitively on them.

(c) The $G$-stabilizer $T$ of any homogeneous component $S$ of
$\res_N^G V$ is the inertia group of any $FN$-irreducible
submodule of $S$.

(d)  $V\cong \ind_T^G S$ and $S$ is $FT$-irreducible.
\end{theorem}

\begin{proof} (a) The sum of irreducible submodules is the direct
sum of some of them.

(b) That $G$ acts follows from Lemma \ref{u1}. The rest follows
from  (a).

(c) This is consequence of the fact that $g^{-1}W\cong_{FN} W^g$
for all $g\in G$.

(d) $V\cong \ind_T^G S$ by (b). This and the irreducibility of $V$
implies that of $S$.
\end{proof}

\begin{theorem}\label{cl2} (Clifford) Let $N\unlhd G$ be groups and
let $W$ be an irreducible $FN$-module with inertia group $T$.
Suppose $S$ is an irreducible $FT$-module lying over $W$. Then $V=\ind_T^G S$
is an irreducible $FG$-module.
\end{theorem}

\begin{proof} Let $U$ be a non-zero $FG$-submodule of $V$. By Theorem \ref{cl1}, $S$ is
a completely reducible $FN$-module, whence so is $V$ and a
fortiori its $FN$-submodule $U$. In particular, $U$ contains an
irreducible $FN$-submodule, necessarily isomorphic to $gW$ for
some $g\in G$. It follows from Lemma \ref{u1} that $g^{-1}U=U$
contains a $FN$-submodule isomorphic to $W$.

Now, the $W$-homogeneous component of the completely reducible
$FN$-module~$V$ is $S$. By above, $U\cap S\neq 0$. Since $S$ is
$FT$-irreducible, $S$ is contained in $U$, and therefore $U=V$.
\end{proof}

\begin{cor}\label{57} If $T=N$ then $\ind_N^G W$ is an irreducible $FG$-module.
\end{cor}

\begin{theorem} (Clifford) Let $N\unlhd G$ be groups and
let $W$ be an irreducible $FN$-module with inertia subgroup $T$. Then the maps
$$
S\to V(S)=\ind_T^G S,
$$
$$
V\to S(V)= \text{$W$-homogeneous component of }V
$$
yield inverse bijections between the isomorphism classes of
irreducible modules of $FT$ and $FG$ lying over $W$.
\end{theorem}

\begin{proof} Immediate consequence of Theorems \ref{cl1} and
\ref{cl2}.
\end{proof}


\begin{note}\label{r1}{\rm It follows from \cite[Lemma 6.1.2]{P} that the above classes are
non-empty.}
\end{note}

\begin{lemma}\label{forma} Let $A$ be an $F$-algebra and let $W$ be an irreducible $A$-module such that
$\End_A(W)=F$. Let $U_1,U_2$ be $F$-vector spaces and view each $U_i\otimes_F W$ as a left $A$-module
via $a(u\otimes w)=u\otimes aw$. Suppose $T\in\Hom_A(U_1\otimes W,U_2\otimes W)$. Then $T=S\otimes 1$,
where $S\in\Hom_F(U_1,U_2)$.
\end{lemma}

\begin{proof} Fix $0\neq w\in W$. We claim that
there is $S_w\in\Hom_{F}(U_1,U_2)$ such that $T(u\otimes w)=S_w(u)\otimes w$ for all
$u\in U_1$. Indeed, for $u\in U_1$ we have
$$
T(u\otimes w)=u_1\otimes w_1+\cdots+u_n\otimes w_n,
$$
where $w_1,\dots,w_n\in W$ are linearly independent over $F$ and
$u_1,\dots,u_n\in U_2$.

\smallskip

\noindent{\sc Case I.} $w,w_1,\dots,w_n$ are
linearly independent. Since $W$ is $A$-irreducible and $\End_A(W)=F$, the density theorem (cf.
\cite[Theorem 2.1.2]{He}), ensures the existence of $r\in A$ such that
$rw=w$ and $rw_i=0$, $1\leq i\leq n$. Thus
$$
0=r(u_1\otimes w_1+\cdots+u_n\otimes w_n)= rT(u\otimes
w)=T(r(u\otimes w))=T(u\otimes w).
$$
\noindent{\sc Case II.} $w,w_1,\dots,w_n$ are linearly dependent. Since $w\neq 0$, we have
$$
T(u\otimes w)=y_1\otimes z_1+\cdots+y_n\otimes z_n,
$$
where $z_1,\dots,z_n\in W$ are linearly independent, $z_1=w$, and
$y_1,\dots,y_n\in U_2$. As before, we can find $a\in A$ such that
$az_1=z_1$ and $az_i=0$ if $i>1$. Then
$$
y_1\otimes z_1=a(y_1\otimes z_1+\cdots+y_n\otimes
z_n)=T(a(u\otimes w))=y_1\otimes z_1+\cdots+y_n\otimes
z_n,
$$
so
$$
T(u\otimes w)=y_1\otimes z_1=y_1\otimes w.
$$

In either case, $T(u\otimes w)=u'\otimes w$, where $u'\in U_2$ is uniquely determined by $w$. This defines a map
$S_w:U_1\to U_2$ satisfying $T(u\otimes w)=S_w(u)\otimes w$ for all $u\in U_1$. Since $T$ is linear
we readily verify that $S_w\in\Hom_{F}(U_1,U_2)$. This proves the claim.

As $T$ is linear and $\otimes$ is bilinear, we see that $S_{w_1}=S_{w_2}$ for all $w_1,w_2\in W$ different from 0.
Call this common operator $S\in\Hom_{F}(U_1,U_2)$. Then $T(u\otimes w)=S(u)\otimes w$ for all $u\in U_1$, $w\in W$ (including $w=0$,
since $0=0w'$ for any $0\neq w'\in W$), whence $T=S\otimes 1$.
\end{proof}

\begin{theorem}\label{cl3} (Clifford) Let $N\unlhd G$ be groups,
let $S:G\to\GL(V)$ be an irreducible representation such that $V$
has an irreducible $G$-invariant $FN$-submodule~$W$ satisfying
$\End_{FN}(W)=F$. Let $R:N\to\GL(W)$ be the representation of $N$
afforded by $W$, that is, $R(x)=S(x)|_W$ for all $x\in N$. Then
there exist a representation $S':G\to\GL(Z)$ equivalent to $S$,
where $Z=U\otimes_F W$, and irreducible projective representations
$X:G\to\GL(W)$ and $Y:G\to\GL(U)$ such that:

(a) $S'(g)=Y(g)\otimes X(g)$ for all $g\in G$.

(b) $X$ extends $R$.

(c) $Y(x)=1$ for all $x\in N$.
\end{theorem}

\begin{proof} Since $W$ is $FN$-irreducible, Theorem \ref{cl1} ensures that the $FN$-module $V$
is the direct sum of $|I|$ copies of $W$ for some set $I$. Let $U$
be an $F$-vector space of dimension $|I|$ and set $Z=U\otimes_F
W$. Then $x\mapsto 1\otimes R(x)$ is a representation of $N$ on
$Z$. By construction, $Z\cong V$ as $FN$-modules. Since the action
of $N$ on $V$ can be extended to $G$, so does the action of $N$ on
$Z$. Thus, there is a representation $S':G\to\GL(Z)$ equivalent to
$S$ such that $S'(x)=1\otimes R(x)$, $x\in N$.

Since $W$ is $G$-invariant, there is a projective representation
$X:G\to\GL(W)$ extending $R$ such that
$$
X(g)R(x)X(g)^{-1}=R(gxg^{-1}),\quad g\in G,x\in N.
$$
On the other hand, since $S'$ is a representation,
$$
S'(g)(1\otimes R(x))S'(g)^{-1}=1\otimes R(gxg^{-1}), \quad g\in G,x\in N.
$$
For $g\in G$ let $P(g)=S'(g)(1\otimes X(g))^{-1}$. Then
$$
P(g)(1\otimes R(x))P(g)^{-1}=1\otimes R(x), \quad g\in G,x\in N.
$$
Since $\End_{FN}(W)=F$, Lemma \ref{forma} ensures that, given any $g\in G$, there exists $Y(g)\in\GL(U)$
such that
$$
P(g)=Y(g)\otimes 1,
$$
that is,
\begin{equation}
\label{syx}
S'(g)=Y(g)\otimes X(g).
\end{equation}

Since $X$ extends $R$, it follows that $P(x)=1$, and hence
$Y(x)=1$, for all $x\in N$. Moreover, since $S'$ is a
representation and $X$ is a projective representation, then
(\ref{syx}) forces $Y$ to be a projective representation as well.
Furthermore, since $S'$ is irreducible, so must be $X$ and $Y$.
\end{proof}

\begin{note}\label{gal0}{\rm It is easy to see (cf. \cite[\S 4]{C}) that $X$ can be chosen
so that its associated factor set $\alpha\in Z^2(G,F^*)$ satisfies
\begin{equation}
\label{fact} \alpha(g_1x_1,g_2x_2)=\alpha(g_1,g_2),\quad
g_1,g_2\in G, x_1,x_2\in N.
\end{equation}
}
\end{note}

\begin{theorem}\label{cl4} (Clifford) Let $N\unlhd G$ be groups,
let $R:N\to\GL(W)$ be an irreducible representation that is
$G$-invariant and satisfies $\End_{FN}(W)=F$, and let
$X:G\to\GL(W)$ be a projective representation extending $R$ with
factor set $\alpha\in Z^2(G,F^*)$ satisfying (\ref{fact}). Then

(a) There exists an irreducible projective representation $Y:G\to\GL(U)$ trivial on $N$ with factor set
$\beta=\alpha^{-1}$. Moreover, for any such $Y$, if we set $Z=U\otimes_F W$ and define $S:G\to\GL(Z)$ by
$S(g)=Y(g)\otimes X(g)$, $g\in G$, then $S$ is an irreducible representation of $G$ and
$\res_N^G Z$ is the direct sum of $\dm_F(U)$ copies of $W$.

(b) For $1\leq i\leq 2$, let $Y_i:G\to\GL(U_i)$ be an irreducible
projective representation trivial on $N$ with factor set $\beta$
and let $S_i:G\to\GL(Z_i)$, $Z_i=U_i\otimes_F W$, be defined by
$S_i(g)=Y_i(g)\otimes X(g)$. Suppose $S_1$ and $S_2$ are
equivalent. Then $Y_1$ and $Y_2$ are strictly equivalent, in the sense that
there is an $F$-linear isomorphism $f:U_1\to U_2$ such that
$f(Y_1(g)u)=Y_2(g)f(u)$ for all $g\in G$ and $u\in U_1$.
\end{theorem}

\begin{proof} (a) It follows from (\ref{fact}) that $\alpha$ gives rise to a factor set $\gamma\in Z^2(G/N,F^*)$.
By Zorn's lemma, the twisted group algebra $F^{\delta}[G/N]$,
$\delta=\gamma^{-1}$, has a maximal left ideal. This proves the
existence of $Y$. For any such $Y$, it is clear that $S$ is a
representation of $G$. Let us verify that $S$ is irreducible.

Let $0\neq v\in V$. Then $v=u_1\otimes w_1+\cdots+u_n\otimes w_n$,
where $w_1,\dots,w_n\in W$ are linearly independent over $F$ and
$u_1,\dots,u_n\in U$ are non-zero. Since $W$ is $FN$-irreducible
and $\End_{FN}(W)=F$, the density theorem (cf. \cite[Theorem
2.1.2]{He}) ensures the existence of $r\in FN$ such that
$rw_1=w_1$ and $rw_i=0$ for $i>1$. {\em Since $N$ acts trivially
on $U$}, it follows that
$$
rv=u_1\otimes w_1.
$$
As $W$ is $FN$-irreducible and $N$ acts trivially on $U$, it
follows that $u_1\otimes W$ is contained in $FG\cdot v$. Let $u\in
U$ be arbitrary. Since $U$ is $F^{\beta}[G]$-irreducible,
$u=su_1$, where $s=\alpha_1g_1+\cdots+\alpha_m g_m\in
F^{\beta}[G]$. Every $F$-subspace $\alpha_i g_i(u_1\otimes
W)=\alpha_ig_iu_1\otimes W$ is contained in $FG\cdot v$, and
therefore so is $u\otimes W$. Since $u$ is arbitrary, we infer
$FG\cdot v=V$. This proves that $S$ is irreducible. By
construction, $\res_N^G Z$ is the direct sum of $\dm_F(U)$ copies
of $W$.

(b) This follows from Lemma \ref{forma} applied to $A=FN$.
\end{proof}

\begin{theorem} (Gallagher) Let $N\unlhd G$ be groups and let $R:N\to\GL(W)$ be an irreducible representation
of $N$ over $F$ satisfying:

$\bullet$ $\End_{FN}(W)=F$.

$\bullet$ $R$ extends to a representation $S:G\to\GL(W)$.

Then

(a) Every irreducible $FG$-module lying over $W$ is isomorphic to
$U\otimes_F W$, where $U$ is an irreducible $FG$-module acted upon
trivially by $N$.

(b) If $U$ is an irreducible $FG$-module acted upon trivially by
$N$, then $V=U\otimes_F W$ is an irreducible $FG$-module.

(c) If $U_1$ and $U_2$ are irreducible $FG$-modules acted upon
trivially by $N$ and $U_1\otimes_F W\cong U_2\otimes_F W$ as
$FG$-modules, then $U_1\cong U_2$ as $FG$-modules.
\end{theorem}

\begin{proof} (a) follows from Theorem \ref{cl3}, while (b) and (c) from Theorem
\ref{cl4}.
\end{proof}

\begin{theorem} Let $G_1$ and $G_2$ are groups and set $G=G_1\times G_2$.

(a) Suppose that $V_1$ and $V_2$ are irreducible modules of $FG_1$
and $FG_2$, respectively, such that $End_{FG_2}(V_2)=F$. Then $V=V_1\otimes_F V_2$ is an irreducible $FG$-module.

(b) Let $V$ be an irreducible $FG$-module with an
irreducible $FG_2$-submodule $V_2$ such that $\End_{FG_2}(V_2)=F$. Then
there is an irreducible $FG_1$-module $V_1$ such that $V\cong
V_1\otimes_F V_2$ as $FG$-modules.

(c)  Let $V_1,W_1$ be irreducible $FG_1$-modules,
let $V_2,W_2$ be irreducible $FG_2$-modules, and suppose $V_1\otimes_F
V_2\cong W_1\otimes_F W_2$ as $FG$-modules. Then $V_1\cong W_1$
and $V_2\cong W_2$.
\end{theorem}

\begin{proof} (a) This follows from Theorem \ref{cl4} applied to $N=G_1$ and $\alpha=1$ .

(b) Immediate consequence of Theorem \ref{cl3}.

(c) $V_1\otimes V_2$ (resp. $W_1\otimes W_2$) is a completely reducible $FG_1$-module with
all irreducible submodules isomorphic to $V_1$ (resp. $W_1$). Since $V_1\otimes
V_2\cong W_1\otimes W_2$ as $FG_1$-modules, we infer $V_1\cong W_1$. Using $G_2$ instead of $G_1$,
we obtain $V_2\cong W_2$.
\end{proof}

\begin{note}{\rm The condition $\End_{FG_2}(V_2)=F$ is essential in parts (a) and (b).
Indeed, let $H$ stand for the real quaternions and set $G_1=G_2=\{\pm 1,\pm i,\pm j,\pm k\}$.
Then $H$ becomes an irreducible module for $G=G_1\times G_2$ over $\R$ by declaring
$$
(x,y)h=xhy^{-1},\quad x\in G_1, y\in G_2, h\in H.
$$
But $\res_{G_1}^G H$ and $\res_{G_2}^G H$ remain irreducible, so $H$ cannot be a tensor product. Moreover,
$V=H\otimes_\R H$ is not irreducible, since $\End_{FG}(V)\cong H\otimes_\R H^{\mathrm{op}}\cong M_4(\R)$ is not a division
ring.
}
\end{note}

\begin{exa} Let $V$ be an infinite dimensional $F$-vector space
of countable dimension, let $G_1=\GL(V)=G_2$ and set $G=G_1\times
G_2$. Then $\End(V)$ is an $FG$-module via $(x,y)f=xfy^{-1}$; the
endomorphisms of $V$ of finite rank form an irreducible
$FG$-submodule $U$ of $\End(V)$; $V$ is an irreducible
$FG_1$-submodule of $U$ satisfying $\End_{FG_1}(V)=F$; $V^*$ is an
irreducible $FG_2$-submodule of $U$; $U\cong V\otimes V^*$ as
$FG$-modules.
\end{exa}

\section{Irreducible $FG$-modules without irreducible $FN$-submodules}\label{sec:ex1}




\begin{lemma}\label{art} Let $R$ be a simple ring with a minimal left ideal $I$. Then $R$ is artinian.
\end{lemma}

\begin{proof} The simplicity of $R$ implies $R=IR$, whence $1=x_1y_1+\cdots+x_n y_n$, where $x_i\in I$ and $y_i\in R$.
Thus the map $I^n\to R$, given by $(z_1,\dots,z_n)\mapsto
z_1y_1+\cdots+z_n y_n$, is an epimorphism of $R$-modules. Since
$I$ is an irreducible $R$-module, $R$ has a finite composition
series as $R$-module and is therefore artinian.
\end{proof}

\begin{prop} Let $R$ be a simple non-artinian ring with center $F$ and unit group~$U$. Assume $R$ is $F$-spanned by $U$. Set
$G=G_1\times G_2$, where $G_1=U=G_2$, and view $V=R$ as an
$FG$-module via
$$
(x,y)v=xvy^{-1},\quad x\in G_1, y\in G_2, v\in V.
$$
Then

(a) $V$ is an irreducible $FG$-module.

(b) $V$ has no irreducible $FG_1$-submodules.

(c) $End_{FG}(V)=F$.

(d) $V\not\cong V_1\otimes_F V_2$ for any $FG_i$-modules $V_i$,
$1\leq i\leq 2$.
\end{prop}

\begin{proof} (a) Use that $R$ is simple and
$F$-spanned by $U$. (b) Use Lemma \ref{art} and that $R$ is
$F$-spanned by $U$ and non-artinian. (c) Use that $R$ is
$F$-spanned by $U$. (d) Suppose $V\cong V_1\otimes_F V_2$ for some
$FG_i$-modules $V_i$, $1\leq i\leq 2$. Since $V$ is
$FG$-irreducible, $V_1$ is $FG_1$-irreducible. Thus $V$ contains
an irreducible $FG_1$-submodule, namely a copy of $V_1$, a
contradiction.
\end{proof}


Rings $R$ satisfying the above conditions can be found
at end of \cite[Ch. 1]{L} based on a given field $K$.
For instance, let $R_i=M_{2^i}(K)$, $i\geq 0$, naturally viewed as a
subring of $R_{i+1}$, and set $R=\underset{i\geq 0}\cup R_i$.
Also, let $V$ be a $K$-vector space of countable dimension, let $I$ be the ideal of $\End(V)$
of all endomorphisms of finite rank and set $R=\End(V)/I$ (see also \cite{Z}). For another example,
suppose $K$ admits an automorphism $\sigma$ of infinite order, and set
$R=K[t,t^{-1};\sigma]$, a skew Laurent polynomial ring.

\section{Examples with $\ind_N^G W$ irreducible and $I_G(W)=G$}\label{sec:ex2}



\begin{exa}\label{inva} Let $F$ be a field with an irreducible
polynomial $f(t)=t^m-a\in F[t]$ of degree $m>1$. Let $n$ be the
order, possibly infinite, of $a\in F^*$ and let $G=\langle
x\rangle$ be a cyclic group of order $mn$ (possibly infinite). Let
$N=\langle x^m\rangle$ and view $W=Fw$ as an $FN$-module via
$x^m\cdot w=aw$. Then $N$ is a proper normal subgroup of $G$, $W$
is a $G$-invariant irreducible $FN$-module, and $V=\ind_N^G W$ is
an irreducible $FG$-module.
\end{exa}

\begin{proof} The matrix of $x$ acting on $V$ relative to the
$F$-basis $w,xw,\dots,x^{m-1}w$ is the companion matrix $C$ of
$f$. Since $f$ is an irreducible polynomial over $F$, it follows
that $V$ is an irreducible $FG$-module.
\end{proof}

$\bullet$ Clearly $\End_{FN}(W)=F$. On the other hand, since the
centralizer of the cyclic matrix $C$ in $M_m(F)$ is $F[C]$, it
follows that $\End_{FG}(V)\cong F[C]\cong F[t]/(f)$, which is a
field extension of $F$ of degree $m=[G:N]$.

$\bullet$ $G$ is not a split extension of $N$. This is obvious if
$G$ is infinite. If $n$ is finite then $\gcd(m,n)>1$, for
otherwise $a$ would be an $m$th power in $F^*$, contradicting the
irreducibility of $f$.

$\bullet$ The action of $N$ on $W$ does not extend to $G$, as this
would violate the irreducibility of $f$.

$\bullet$ Example \ref{inva} applies to any field $F$ admitting a
radical extension. This is fairly wide class of rings, as
evidenced by Capelli's theorem (cf. \cite[Theorem 9.1]{La}).
Naturally, algebraically closed fields are excluded. The only
non-allowed finite field is $F_2$.

\begin{exa}\label{inva2} Let $K$ be a field, let $\Gamma$ be a
subgroup of $\Aut(K)$ and let $F$ be a subfield of the fixed field
of $\Gamma$. Given $f\in Z^2(\Gamma,K^*)$, let $G$ be the
extension of $N=K^*$ by $\Gamma$ determined by the factor set $f$.
Then $W=K$ is a $G$-invariant irreducible $FN$-module, and
$V=\ind_N^G W$ is an irreducible $FG$-module if and only if the
crossed product $(K,\Gamma, f)$ is a division ring. In any case,
$\End_{FN}(V)\cong (K,\Gamma, f)$.
\end{exa}

\begin{proof} Every element of $G$ has the form $x_\sigma k$ for
unique $\sigma\in \Gamma$ and $k\in K^*$, with multiplication
$$
x_\sigma k\cdot x_\tau \ell=x_{\sigma\tau}(f(\sigma,\tau) k^\tau
\ell).
$$
On the other hand, every element of $V$ has form
$\underset{\sigma\in \Gamma}\sum x_\sigma k_\sigma$ for unique
$k_\sigma\in K$, almost all zero. The $F$-linear action of $G$ on
$V$ coincides with the multiplication in~$V$, viewed as a crossed
product algebra. Thus $\End_{FN}(V)\cong (K,\Gamma, f)$ and the
$FG$-module $V$ is irreducible if and only if  $(K,\Gamma, f)$ has
no proper left ideals but zero.

For $\sigma\in\Gamma$, the map $W\to x_\sigma W$ given by
$w\mapsto x_\sigma w^\sigma$ is an isomorphism of $FN$-modules, so
$W$ is $G$-invariant. It is obviously $FN$-irreducible.
\end{proof}

Conditions for a cyclic algebra to be a division ring can be found
in \cite[\S 14]{L}.

\begin{theorem} (Frobenius reciprocity) Let $S$ be a ring with a
subring $R$ such that $S=R\oplus T$ as $R$-modules. Let $W$ be an
$R$-module and let $V$ be an $S$-module. Then
$$
\Hom_R(W, \res_R^S V)\cong \Hom_S(\ind_R^S W, V)
$$
as $Z(S)\cap R$-modules, where $\ind_R^S W=S\otimes_R W$.
\end{theorem}

\begin{proof} Given $f\in \Hom_R(W, \res_R^S V)$ define $\widehat{f}\in \Hom_S(\ind_R^S W, V)$ by
$$
\widehat{f}(s\otimes x)=sf(x),\quad s\in S, x\in W.
$$
This an isomorphism of $Z(S)\cap R$-modules whose inverse is the
restriction map.
\end{proof}

Let $D$ be a division ring. By an extension of $D$ we mean a
division ring $E$ such that $D$ is a subring of $E$. The degree
of $E$ over $D$ is $\dm_D E$ as right $D$-vector space.

\begin{prop}\label{narr} Let $F$ be a field, let $N\unlhd G$ be groups, let $W$ be an $FN$-module (necessarily irreducible)
with inertia group $T$ such that $V=\ind_N^G W$ is
$FG$-irreducible. Set $U=\Hom_{FN}(W,V)$,
$D=\End_{FN}(W)$ and $E=\End_{FG}(V)$, and consider the $F$-linear isomorphism $U\to E$, $u\mapsto\widehat{u}$, arising
from Frobenius reciprocity, namely $\widehat{u}(g\otimes
w)=g\otimes u(w)$. Then $T=N$ provided at least one of the
following conditions hold:

(a) The restriction of $U\to E$ to $D$ is surjective.

(b) $[G:N]$ is finite and $D$ has no extensions of degree $>1$ and
dividing $[G:N]$.

(c) $D=F$ is algebraically closed and $[G:N]$ is finite.

(d) $G$ is finite and $F$ is algebraically closed.

(e) The action of $N$ on $W$ can be extended to $T$.

(f) $\dm_F(W)=1$ and $T$ is a split extension of $N$.
\end{prop}

\begin{proof} $U$ is a right $D$-vector space via $u\cdot d=u\circ d$. As the multiplicity of $W$ in $\res^G_N V$ is $[T:N]$, we see 
that $\dm_D U=[T:N]$. Since the restriction of $U\to E$ to $D$ is a homomorphism of
$F$-algebras, it follows that $E$ is right $D$-vector space via
$e\cdot d=u\circ \widehat{d}$. By definition, $U\to E$ is an
isomorphism of right $D$-vector spaces and the inclusion map $D\to
U$ is a monomorphisms of right $D$-vector spaces. This yields (a)
and (b), which implies (c), which implies (d).

Suppose the action of $N$ on $W$ can be extended to $T$. Since
$\ind_N^G W$ is irreducible, it follows that $U=\ind_N^T W$ is
irreducible. Let $f:W\to \res_N^T W$ be the identity map and let
$\widehat{f}:U\to W$ be its extension to an $FT$-homomorphism.
Clearly, $\widehat{f}$ is surjective and hence bijective, since
$U$ is irreducible. But the restriction of $\widehat{f}$ to the
$FN$-submodule $W$ of $U$ is already bijective, which forces
$T=N$. This proves~(e), which easily implies (f). 
\end{proof}

Example \ref{inva} avoids conditions (a)-(f) of Proposition
\ref{narr}. Consider the following special case of Example
\ref{inva2}: $F=\C$, $K=\C(s,t)$, $\Gamma=\langle \gamma\rangle$,
where $\gamma(s)=s$ and $\gamma(t)=-t$, and $f$ is the normalized
factor set determined by $f(\gamma,\gamma)=s$. Since $s$ not a
square in $\C(s)$, it follows from \cite[\S 14]{La} that
$(K,\Gamma, f)$ is a 4-dimensional division algebra over its
center $\C(s,t^2)$. Thus $V$ irreducible, $[G:N]=2$ and $F=\C$,
still avoiding condition (c) of Proposition \ref{narr} because
$\dm_\C(D)$ is infinite.

An improvement of Proposition \ref{narr} may be available by
suitably extending to arbitrary groups \cite[Theorem 1 and
Corollary 1]{T3} of Tucker, which describes $\End_{FG}(V)$ as a
crossed product when $G$ is finite, as well as extending to
arbitrary fields her correspondence (cf. \cite{T, T2} and Conlon's
\cite{C}) between the left ideals of $\End_{FG}(V)$ and the
$FG$-submodules of $V$ when $F$ is algebraically closed.

\end{document}